\documentclass[12pt,a4paper,reqno]{amsart}
\usepackage{amssymb,hyperref}

\providecommand{\BBb}[1]{{\mathbb{#1}}}
\providecommand{\cal}[1]{{\mathcal{#1}}}   
\newcommand{\C}{{\BBb C}}
\newcommand{\dual}[2]{\langle\,#1,\,#2\,\rangle}
\newcommand{\fracc}[2]{{
                \textstyle\frac{#1}{\raise 1pt\hbox{$\scriptstyle #2$}}}}

\newcommand{\fracnp}{\fracc np}

\newcommand{\fracci}[2]{{\frac{#1}{\raise 1pt\hbox{$\scriptscriptstyle #2$}}}}
\newcommand{\fracpi}{\fracci1p}

\newcommand{\im}{\operatorname{i}}
\newcommand{\imb}{\hookrightarrow}
\newcommand{\loc}{\operatorname{loc}}

\newcommand{\nrm}[2]{\|#1\|_{#2}}
\newcommand{\Nrm}[2]{\bigl\|#1\bigr\|_{#2}}
\newcommand{\norm}[2]{\mathinner{\|}#1\,|#2\|}
\newcommand{\Norm}[2]{\mathinner{\bigl\|\,#1\,\big|#2\bigr\|}}
\newcommand{\N}{\BBb N}
\newcommand{\op}[1]{\operatorname{#1}}
\newcommand{\R}{{\BBb R}}
\newcommand{\Rn}{{\BBb R}^{n}}
\newcommand{\supp}{\operatorname{supp}}
\newcommand{\Z}{\BBb Z}

\newtheorem{thm}{Theorem}
\newtheorem{prop}[thm]{Proposition}
\newtheorem{lem}[thm]{Lemma}

\theoremstyle{definition}

\theoremstyle{remark}
\newtheorem{rem}[thm]{Remark}
\title[Sobolev embeddings for Lizorkin--Triebel spaces]{%
A direct proof of Sobolev embeddings for 
quasi-homogeneous Lizorkin--Triebel spaces with mixed norms}
\author[J.~Johnsen]{Jon Johnsen}
\address{\small Department of Mathematical Sciences,
 Aalborg University,
 Fredrik~Bajers Vej 7G,
 DK--9220 Aalborg {\O}st, Denmark }
\email{jjohnsen@math.aau.dk}
\author[W.~Sickel]{Winfried Sickel}
\address{\small Mathematics Department,
 Friedrich-Schiller-University Jena,
 Ernst-Abbe-Platz 2,
 D--07743 Jena, Germany} 
\email{sickel@minet.uni-jena.de}

\begin{document}

 \begin{abstract}
The article deals with a simplified proof of the Sobolev embedding
theorem for Lizorkin--Triebel spaces (that contain the $L_p$-Sobolev
spaces $H^s_p$ as special cases). The method extends to a proof of the
corresponding fact for general Lizorkin--Triebel spaces based
on mixed $L_p$-norms. In this context a
Nikol$'$skij--Plancherel--Polya inequality  for sequences of functions
satisfying a geometric rectangle condition is proved. The results extend
also to anisotropic spaces of the quasi-homogeneous type.
 \end{abstract}
\maketitle
\section{Introduction}
To give an overview, we first comment on standard Lizorkin--Triebel spaces 
(i.e.\ isotropic, inhomogeneous spaces with unmixed norms). These are
throughout denoted by $F^{s}_{p,q}$.

\bigskip

Since around 1977 the question of Sobolev
embeddings of Lizorkin--Triebel spaces has been answered affirmatively,
with a unified proof of the following

\begin{prop}
  \label{sob-prop}
Let $s$, $s_0\in \R$ and $p$, $p_0\in \,]0,\infty[$ be given such
that
\begin{equation}
  s-\fracnp=s_0-\fracc n{p_0},\quad p>p_0.
  \label{sspp-ineq}
\end{equation}
There is then a continuous embedding
$F^{s_0}_{p_0,\infty}(\Rn)\imb F^{s}_{p,q}(\Rn)$ for every $q\in \,]0,\infty]$.
\end{prop}

Specifically this means that there exists a number $C>0$, depending
on the parameters, such that the following inequality is valid for every
$u\in F^{s_0}_{p_0,\infty}(\Rn)$:
\begin{equation}
  \bigl(\int_{\Rn} (\sum_{j=0}^\infty
  2^{sjq}|u_j(x)|^q)^{\frac{p}{q}}\,dx\bigr)^\fracpi
\le C
  \bigl(\int_{\Rn} (\sup_{j\in \N_0}
  2^{s_0j}|u_j(x)|)^{p_0}\,dx\bigr)^\fracci1{p_0}.
  \label{sob-ineq}
\end{equation}

To explain the notation,
note that a Littlewood--Paley decomposition $1=\sum_{j=0}^\infty \Phi_j(\xi)$ 
can be obtained by letting $\Phi_j=\Psi_j-\Psi_{j-1}$ for $j>0$ when
$\Psi_j(\xi)=\psi(2^{-j}|\xi|)$ for some $\psi\in  C^\infty(\R,\R)$
fulfilling $\psi(t)=1$ for $t<1$ and $\psi(t)=0$ for $t>2$. 

By definition, a tempered distribution $u\in \cal
S'(\Rn)$ is in the Lizorkin--Triebel space $F^{s}_{p,q}(\Rn)$
when the quantity on the left hand side of \eqref{sob-ineq} is
finite; hereby $u_j=\cal F^{-1}(\Phi_j\hat u)$. 
As usual $\nrm{f}{p}$ denotes the norm of a function $f$ in $L_p(\Rn)$ with
respect to the Lebesgue measure $\lambda$ on $\Rn$. 

Jawerth's original proof \cite{J} of the Sobolev inequalities 
\eqref{sob-ineq} was given for the
homogeneous spaces $\dot F^s_{p,q}$, utilising the rewriting of
$L_p$-norms via the distribution function ($0<p<\infty$) 
\begin{equation}
  \int_{\Rn}|f(x)|^p\,dx=p\int_0^\infty 
  t^{p-1}\lambda(\{\,x\mid |f(x)|>t \,\})\,dt.  
\end{equation}
However, readers familiar with the subject may recall that for the
inhomogeneous $F^{s}_{p,q}$, 
it is not possible for every $t>0$ to find $j\in \N_0$ such that $2^{jn/p}<t$
(unlike $\dot F^{s}_{p,q}$, where $j\in\Z$), so
Jawerth's proof needs an adaptation to this case.
Triebel \cite{T0,T2} introduced a splitting into the
$t$-intervals $\,]0,K[\,$ and $\,]K,\infty[\,$, where 
$K$ essentially was the $q^{\op{th}}$ root of $\norm{u}{F^{s}_{p,q}}$.

Since then this proof has widely been considered
`best possible'. Perhaps this is because the strategy of Jawerth and Triebel
covers all $p\in \,]0,\infty[\,$ in an elegant way (at least for $\dot
F^{s}_{p,q}$), whilst previous attempts did not cover all cases.

In comparison, the Besov spaces $B^{s}_{p,q}(\Rn)$ have 
corresponding embeddings with a well-known one-line proof based on the
Nikol$'$skij--Plancherel--Polya inequality that we now recall.

For $0<p\le r\le\infty$ there exists a $c>0$
such that for every
$f\in \cal S'(\Rn)\cap L_p$ with $\supp\cal Ff$ compact, say
contained in the closed ball $B(0,R)=\{\,\xi\in\Rn \mid |\xi|\le R\,\}$,
\begin{equation}
  \nrm{f}{r}\le c R^{\fracci np-\fracci nr}\nrm{f}{p}.
  \label{nik-ineq}
\end{equation}
Applying this to each $u_j$ ($R=2^{j+1}$) in $\norm{u}{B^{s}_{p,q}}:=(\sum
2^{sjq}\nrm{u_j}{p}^q)^{\fracci1q}$, one finds at once that
$B^{s_0}_{p_0,q}\imb B^{s}_{p,q}$ when \eqref{sspp-ineq} holds.

Generally $F^{s}_{p,q}$-spaces are rather more
complicated to treat than $B^{s}_{p,q}$-spaces, so it would seem plausible
that even the basic Sobolev inequalities in \eqref{sob-ineq} are
substantially more technical to achieve for the $F^s_{p,q}$.

But 
also \eqref{sob-ineq} has a short proof based on the
Nikol$'$skij--Plancherel--Polya inequality \eqref{nik-ineq}.
The trick is to handle the infinite sum by means of the following 
result on the weighted sequence spaces
$\ell^s_q$, that (for $q<\infty$) is normed by  
\begin{equation}
  \norm{a_j}{\ell^s_q}= (\sum_{j=0}^\infty 2^{sjq}|a_j|^q)^{\fracci1q}.
\end{equation}

\begin{lem}
  \label{sss-lem}
Let real numbers $s_1<s_0$ be given, and $\theta\in \,]0,1[\,$.
For $0<q\le\infty$ there is $c>0$ such that
\begin{equation}
  \norm{a_j}{\ell^{\theta s_0+(1-\theta)s_1}_q}\le c
    \norm{a_j}{\ell^{s_0}_\infty}^\theta
    \norm{a_j}{\ell^{s_1}_\infty}^{1-\theta}
  \label{sss-ineq}
\end{equation}
holds for all complex sequences $(a_j)_{j\in \N_0}$ in $\ell^{s_0}_\infty$.
\end{lem}

This result was to our knowledge first crystallised in the works of Brezis
and Mironescu \cite{BrMi01}; cf.\ the elementary proof there.

In what follows we shall discuss certain generalizations 
of Proposition~\ref{sob-prop}, and of \eqref{nik-ineq},
to mixed $L_p$-norms and anisotropic smoothness.
For the reader's convenience we shall first give the details of the proof 
in the isotropic (unmixed) case. Hence we 
prove Proposition \ref{sob-prop};
as it will become clear in Sections~2 and 3, 
the method carries over straightforwardly to the more general situation.

\begin{proof}[Proof of Proposition~1 
(based on \eqref{nik-ineq}, \eqref{sss-ineq})]
For $u\in F^{s_0}_{p_0,\infty}$ the claim is obtained 
by interpolating its $F^{s}_{p,q}$-norm between those of
$B^{s_1}_{\infty,\infty}$ with
$s_1:=s_0-\fracc n{p_0}=s-\fracnp$ and the given space 
$F^{s_0}_{p_0,\infty}$:
 
Since $p>p_0$ there is some
$\theta\in \,]0,1[\,$ so that $\theta\fracc n{p_0}=\fracnp$, hence
\begin{equation}
  \theta s_0+(1-\theta)s_1=\theta s_0+(1-\theta)(s_0-\fracc n{p_0})
 = s-\fracnp+\theta\fracc n{p_0}=s.
\end{equation}
Clearly $q<\infty$ suffices in \eqref{sob-ineq}, 
and for each $x\in \Rn$, Lemma~\ref{sss-lem} gives
\begin{equation}
  (\sum_{j=0}^\infty 2^{sjq}|u_j(x)|^q)^{\fracci1q}
\le c
  \norm{u_j(x)}{\ell^{s_0}_\infty}^\theta
  (\sup_{j}2^{s_1j}\nrm{u_j}{\infty})^{1-\theta}.
  \label{aux-ineq}
\end{equation}
Here $\nrm{u_j}{\infty}\le c'2^{j\fracci n{p_0}}\nrm{u_j}{p_0}$ follows from
\eqref{nik-ineq} since  $u_j$ is in $L_{p_0}$ with 
$\supp\cal Fu_j\subset B(0,2^{j+1})$.
The definition of $s_1$ therefore yields
\begin{equation}
  \sup_{j}2^{s_1j}\nrm{u_j}{\infty}\le
  c'  \sup_{j}2^{(s_1+\fracci n{p_0})j}\nrm{u_j}{p_0}
  \le c' \norm{u}{F^{s_0}_{p_0,\infty}}.
\end{equation}
Hence the fact that $\theta p=p_0$ gives,
by taking $L_p$-norms in \eqref{aux-ineq}, 
\begin{equation}
  \norm{u}{F^s_{p,q}}\le c \norm{u}{F^{s_0}_{p_0,\infty}}^\theta
  (c'\norm{u}{F^{s_0}_{p_0,\infty}})^{1-\theta}=c''
  \norm{u}{F^{s_0}_{p_0,\infty}}.
\end{equation}
That is, \eqref{sob-ineq} is proved.
\end{proof}

\begin{rem}
For simplicity we have taken the sum-exponent $q_0=\infty$ on the
right hand side of \eqref{sob-ineq}, but 
the other Sobolev inequalities can be recovered via the simple embedding
$F^{s_0}_{p_0,q_0}\imb F^{s_0}_{p_0,\infty}$.
\end{rem}

The Nikol$'$skij--Plancherel--Polya
inequality \eqref{nik-ineq} will be generalised to mixed $L_p$-norms
in Section~\ref{mixd-sect} below.
In the course of the proof given there, 
we utilise \eqref{nik-ineq} in the above unmixed form; 
so to provide a complete overview we insert an argument for
this (known) standard version.

\begin{proof}[Proof of inequality \eqref{nik-ineq}]
It suffices to establish \eqref{nik-ineq} when in addition to the stated
assumptions $u\in\cal S(\Rn)$, where $\cal S(\Rn)$ is the Schwartz space of
rapidly decreasing $C^{\infty }$-functions. 
To see this we take $\hat\psi\in
C^\infty(\Rn)$ such that $\hat\psi(t)=1$ for $|t|\leq1$,
$\hat\psi(t)=0$ for $|t|\geq2$, while $\psi(0)=1$. Defining
$u_j$ by
\begin{equation}
  u_j(x):= u(x)\psi(\tfrac xj),
\end{equation}
then $u_j$ 
has its spectrum in $\supp\hat u+B(0,\tfrac{2}{j})$, 
and $u_j\in \cal S(\Rn)$ for
$u(x)$ is $\cal O((1+|x|)^N)$ for some $N\geq0$ by the
Paley--Wiener--Schwartz theorem ($\hat u\in \cal E'$).
There is only something to show for $p<r$, so $p<\infty$ and hence
\begin{equation}
  u_j\to u\quad\text{in $\cal S'(\Rn)\cap L_{p}(\Rn)$}.
  \label{uju-lim}
\end{equation}
So if \eqref{nik-ineq} can be proved for all Schwartz functions with
compact spectra, it will follow
that $(u_j)$ is a fundamental sequence in $L_{r}$, hence
converging to some $g\in L_{r}$
with $g(x)=u(x)$ a.e.

Since $\nrm{\cdot}{p}$ is continuous $L_{p}\to\R$, the above
would entail that 
\begin{equation}
  \norm{u}{L_{r}} 
  \leq 
  \lim_{j\to\infty} c(R+\tfrac2j)^{\fracci np-\fracci nr}\norm{u_j}{L_{p}}   
  =
  cR^{\fracci np-\fracci nr}\norm{u}{L_{p}}.
\end{equation}
($c>0$ will be independent of
$u$ and the size of its spectrum.)

In the smooth case we proceed in the spirit of \cite[1.3.2]{T2}.
When $u\in \cal S(\Rn)$ with $\hat u\in C^\infty_{B(0,R)}$, then 
$\Psi(x)=R^n\psi(Rx)$ fulfils
\begin{equation}
  u(x)=\cal F^{-1}(\hat\Psi\hat u)=\int_{\Rn} u(y)\Psi(x-y)\,dy.
  \label{upsi-eq}
\end{equation}
For $1\leq p\leq\infty=r$ and
$p+p^*=pp^*$, the above identity gives
\begin{equation}
  |u(x)|\leq \norm{\Psi}{L_{p^*}}\norm{u}{L_{p}}
        \leq \norm{\psi}{L_{p^*}} 
        R^{\fracci n{p}} \norm{u}{L_{p}},
  \label{upp'-eq}
\end{equation}
hence \eqref{nik-ineq}. In case $0<p<1$ and $r=\infty$,
\begin{equation}
  |u(x)|\leq  \nrm{\Psi}{\infty}
  (\sup|u(z)|)^{1-p} \int |u(y)|^p\,dy,
\end{equation}
so the $x$-independence of the right hand side entails
\begin{equation}
  |u(x)|^p\leq \nrm{u}{\infty}^p\leq \nrm{\Psi}{\infty} \int |u(y)|^p\,dy,
  \label{up1-eq}
\end{equation}
where $\nrm{\Psi}{\infty}^{\fracci1p}=
R^{\fracci np}\nrm{\psi}{\infty}$. This proves \eqref{nik-ineq} for $r=\infty$.

If $r<\infty$ and $p<1$, insertion of \eqref{up1-eq} into $\nrm{u}{r}$ gives
\begin{equation}
\begin{split}
(\int |u(x)|^r\,dx)^{\fracci1r} &\leq
(\sup_z|u(z)|^{r-p})^{\fracci1r}(\int |u(x)|^p\,dx)^{\fracci1r} 
\\  &\leq
 \nrm{\Psi}{\infty}^{\fracci{r-p}{rp}} \nrm{u}{p}
\le
 (\nrm{\psi}{\infty} R^n)^{\fracci{1}{p}-\fracci1r} \nrm{u}{p}.
\end{split}
\end{equation}
For $1\leq p\leq r<\infty$ one can insert
\eqref{upp'-eq} instead.
This completes the proof of \eqref{nik-ineq},
with $c$ equal to a power of a norm of $\psi$.
\end{proof}

\section{The mixed norm case}
  \label{mixd-sect}
One advantage of the above proof method is that it extends to spaces with
mixed norms. In this set-up $L_p(\Rn)$ is replaced by $L_{\vec p}(\Rn)$
where
$\vec p=(p_1,\dots,p_n)$ for $0<p_j\le\infty$, $j=1,\dots,n$, and
\begin{equation}
  \norm{f}{L_{\vec p}}=
  (\int_{\R}(\dots 
  (\int_\R |f(x_1,\dots,x_n)|^{p_1}\,dx_1
     )^{\frac{p_2}{p_1}}
  \dots     )^{\frac{p_n}{p_{n-1}}}\,dx_n)^\fracci1{p_n}.
\end{equation}
It is well known that such spaces frequently enter the analysis of 
evolution problems for partial differential equations.

The corresponding Lizorkin--Triebel spaces $F^{s}_{\vec p,q}(\Rn)$ consist
of the $u$ in $\cal S'(\Rn)$ having finite quasi-norms ($0<p_j<\infty$)
\begin{equation}
  \norm{u}{F^{s}_{\vec p,q}} = 
   \norm{(\sum_{j=0}^\infty 2^{sqj}|u_j(\cdot)|^q)^\fracci1q}{L_{\vec p}}.
  \label{fspq'-eq}
\end{equation}
The purpose is not to go into the general theory of such spaces here.
Instead we want to show that
the Sobolev embeddings follow directly from first principles, namely the
definition \eqref{fspq'-eq} and a mixed
norm version of the Nikol$'$skij--Plancherel--Polya inequality.

Before we turn to this, it is noted that also the mixed quasi-norm in
$L_{\vec p}$ is subadditive when raised to a power 
$\lambda\in\,]0,\min(1,p_1,\dots,p_n)]$,
\begin{equation}
  \norm{f+g}{L_{\vec p}}^{\lambda}
\leq
    \norm{f}{L_{\vec p}}^{\lambda}+  \norm{g}{L_{\vec p}}^{\lambda}.
\label{Lplambda-eq}
\end{equation}
In fact $\int |f+g|^p\,d\mu\leq \int |f|^p\,d\mu+\int|g|^p\,d\mu$ holds
for arbitrary measures if $0<p\leq 1$, so 
$\nrm{f+g}{p}\leq (\nrm{f}{p}^\nu+\nrm{g}{p}^\nu)^{1/\nu}$ 
for  \emph{all} $\nu\in\,]0,\min(1,p)]$. 
Using this, it is easy to see that $\Nrm{ \nrm{\cdot}{q} }{p}^\lambda$
is subadditive for $\lambda\leq\min(1,p,q)$. For
$\lambda=\min(1,p_1,\dots,p_n)$ a repeated use of this yields
\eqref{Lplambda-eq}. 

To prepare for the mixed norm version of the
Nikol$'$skij--Plancherel--Polya inequality, we recall the
Paley--Wiener--Schwartz Theorem in a precise version, cf.~\cite[Ch.~7]{H}:
when $K\Subset \Rn$, a distribution $u\in \cal S'(\Rn)$ fulfils
$\supp\hat u\subset K$ if and only if $u$ extends to a function 
$u(x+\im y)$ on $\C^n$ which is entire analytic and fulfils
\begin{equation}
 |u(x+\im y)|\le C(1+|x+\im y|)^N e^{H(-y)}, \quad x,y\in \Rn 
  \label{ON-eq}
\end{equation}
for some $N\ge0$ (the order of $\hat u$);
here $H(y)=\sup\{\,y\cdot\xi\mid  \xi\in K\,\}$
is the supporting function of the compact set $K$. 

When $u\in \cal S'\cap L_{\vec p}(\Rn)$ and $\supp\hat u\Subset\Rn$, then
$u(x)=\cal O((1+|x|)^N)$  for $x\in \Rn$, so
$u\in L_1^{\loc}(\Rn)$ also if some $p_k<1$; 
and \eqref{ON-eq} also yields $u(\varphi)=\int
u(x)\varphi(x)\,dx$ for $\varphi\in \cal S(\Rn)$. 
(In general elements of $\cal S'\cap L_1^{\loc}$
do not act on $\cal S$ by integration; cf. $e^x\cos(e^x)=\sin(e^x)'$ on $\R$.)

We also note that for a rectangle $[-R_1,R_1]\times\dots\times[R_n,R_n]$,
\begin{equation}
  H(y)= R_1|y_1|+\dots+R_n|y_n|.
\end{equation}
Indeed, $H(y)$ can be estimated by the triangle inequality, and equality is
attained in one of the corners $(\pm R_1,\dots,\pm R_n)$.
When $\hat u$ is supported in this rectangle, and $(x',x'')$, $(y',y'')\in
\R^{n'}\times\R^{n''}$, $n=n'+n''$, it follows when $x''$, $y''$ are kept as
parameters that $U(x'+\im y'):= u((x',x'')+\im(y',y''))$ is analytic on
$\C^{n'}$ and
\begin{equation}
  |U(x'+\im y')|\le C' (1+|(x',y')|)^N \exp(\sum_{n'}R_j|y_j|).
  \label{ON'-eq}
\end{equation}
Therefore $x'\mapsto U(x')$ is a tempered distribution on $\R^{n'}$ with
spectrum in $\prod_{n'}[-R_j,R_j]$, as one would expect.

These facts are convenient for the proof of

\begin{prop}
  \label{mNPP-prop}
Whenever $0<p_j\le r_j\le\infty$ for $j=1,\dots,n$, there is a $c>0$
such that for every $f\in \cal S'(\Rn)\cap L_{\vec p}$ with spectrum
in a compact rectangle given by $|\xi_k|\leq R_k$ for $k=1,\dots,n$,
i.e.\ 
\begin{equation}
  \supp\cal Ff\subset [-R_1,R_1]\times\dots[-R_n,R_n],
\label{R-eq}
\end{equation} 
it also holds that $f\in L_{\vec r}(\Rn)$ and
\begin{equation}
  \norm{f}{L_{\vec r}}\le c (\prod_{k=1,\dots,n}
             R_k^{\fracci1{p_k}-\fracci1{r_k}})\norm{f}{L_{\vec p}}.
  \label{mNPP-ineq}
\end{equation}
\end{prop}
This result was established by Unin$'$skij \cite{Un66,Un70} for
exponents $p_k\ge1$; Schmeisser and Triebel \cite[1.6.2]{ScTr87} 
covered the case $n=2$. 
We give a direct proof where the treatment of $0<p_j<1$ is inspired by 
a paper of St{\"o}ckert \cite[Satz~2.1]{Stck78}, who proved
\eqref{mNPP-ineq} for subclasses of $\cal S(\Rn)$ with exponential decay.

The strategy of the proof is perhaps best described as a \emph{succession of
embeddings}, which means that \eqref{mNPP-ineq} is realised as a composition
of the following $n$ embeddings, that each only affect a single coordinate
direction:
\begin{equation}
  L_{\vec p}\imb L_{(r_1,p_2,\dots,p_n)}
  \imb L_{(r_1,r_2,p_3,\dots,p_n)}\imb\dots\imb L_{\vec r}.
\end{equation}
For $r_k=p_k$ the $k^{\op{th}}$ map is just the identity map from the
$k^{\op{th}}$ space to itself. This gives a convenient reduction
to the case in which $\vec r$ and $\vec p$ differ in only one component, say
$r_m>p_m$ whilst $r_j=p_j$ for $j\ne m$. Thereby some technicalities are
circumvented. (Already \eqref{uju-lim} is troublesome to carry over to the
mixed case, for if some $p_k=\infty$ we cannot obtain convergence in the norm
topology.)

\begin{proof}
$1^\circ$. We prove \eqref{mNPP-ineq} for an arbitrary $u\in 
L_{\vec p}(\Rn)$ 
by means of a succession of
embeddings, as explained above. So it suffices to
assume that $r_m>p_m$ only holds for one value of $m$, and we can assume
that this is $m=n$. 
 
Indeed, seeing $u$ as a function of $x'=(x_1,\dots,x_m)$, it was found
in \eqref{ON'-eq} ff.\ that its spectrum lies in
$[-R_1,R_1]\times\dots\times [-R_m,R_m]$; and by \eqref{ON-eq}
it is in $\cal S'\cap L_{(p_1,\dots,p_m)}$ on $\R^{m}$, 
so once the case $m=n$ is covered,
integration with respect to $x''$ yields
\eqref{mNPP-ineq} in general when $1\le m\le n$.

$2^\circ$.
If $\psi\in\cal S(\R)$ is such that $\hat\psi(t)=0$ for $|t|\ge2$ while
$\hat\psi(t)=1$ for $|t|\le1$, we set $\Psi(x_n)=R_n\psi(R_nx_n)$. 
Clearly $\hat u(\xi)=\hat\Psi(\xi_n)\hat u(\xi)$, and we want to
show that
\begin{equation}
  u(x)=\int_{\R} u(x',y_n)\Psi(x_n-y_n)\,dy_n\quad\text{for}\quad x\in \Rn.
  \label{uPsi-eq}
\end{equation}
The right hand side makes sense by \eqref{ON-eq}, and for $\varphi\in
C^\infty_0(\Rn)$ 
\begin{equation}
  \cal F_{x\to\xi}(\int \varphi(x',y_n)\Psi(x_n-y_n)\,dy_n)
  = \hat\Psi(\xi_n)\hat\varphi(\xi).
\end{equation}
From this and the inversion formula $\cal
F^2\varphi=(2\pi)^n\varphi(-\cdot)$ used twice,
\begin{equation}
  \begin{split}
  \dual{u}{\varphi(-\cdot)}&=
  \dual{\hat u}{\hat \varphi}/(2\pi)^n
  =  \dual{\hat u}{\hat\Psi(\xi_n)\hat\varphi(\xi)}/(2\pi)^n
  \\
  &=\dual{u}{\cal F_{\xi\to x} (\hat\Psi(\xi_n)\hat\varphi(\xi))}
     /(2\pi)^n
  \\
  &=\iiint u(x)\Psi(y_n-x_n)\varphi(-x',-y_n)\,dx_ndy_ndx'.
  \end{split}
\end{equation}
The last identity follows from Fubini's theorem and the polynomial
growth. In particular $(x',y_n)\mapsto \int u(x)\Psi(y_n-x_n)\,dx_n
\varphi(-x',-y_n)$ is integrable, and since $u$ itself acts by integration,
this shows \eqref{uPsi-eq} in the set of locally integrable functions.

$3^\circ$. In case $p_k\ge1$ for all $k$ it follows from \eqref{uPsi-eq}
and the generalised Minkowski inequality that
\begin{equation}
  \norm{u(\cdot,x_n)}{L_{p'}(\R^{n-1})}\le
  \int |\Psi(x_n-y_n)|\cdot
   \norm{u(\cdot,y_n)}{L_{p'}}\,dy_n.
\end{equation}
The Hausdorff--Young inequality applies to the convolution in the last
expression, so if $q\ge1$ is defined by
$\fracc1{p_n}+\fracc1q=1+\fracc1{r_n}$,
\begin{equation}
  \norm{u}{L_{\vec r}}\le \norm{\Psi}{L_q(\R)}\norm{u}{L_{\vec p}}
  \le R_n^{1-\fracci1q}\nrm{\psi}{q}\norm{u}{L_{\vec p}}.
\end{equation}
Because $1-\fracc1q=\fracc1{p_n}-\fracc1{r_n}$, this yields
\eqref{mNPP-ineq} for all classical exponents.

$4^\circ$. In general there is some $s\in \,]0,1[\,$ 
such that $\tfrac{p_k}{s}>1$ for
all $k$. With this \eqref{uPsi-eq} can be replaced by 
\begin{equation}
  |u(x)|^s\leq C(R_n)\int_{\R} |u(x',y_n)|^s|\Psi(x_n-y_n)|^s\,dy_n.
  \label{upsi-ineq}
\end{equation}
Indeed, $y_n\mapsto u(x',y_n)\Psi(x_n-y_n)$ has
spectrum in $[-3R_n,3R_n]$  
(since that of $u(x',\cdot)$ is contained in $[-R_n,R_n]$).
Hence \eqref{upsi-ineq} results from the next inequality, that follows 
from $2^\circ$ and the standard Nikol$'$skij--Plancherel--Polya
inequality \eqref{nik-ineq} on $\R$,
\begin{equation}
  |u(x)|\leq \nrm{u(x',\cdot)\Psi(x_n-\cdot)}{1}\leq
  c(3R_n)^{\fracci1s-1}\nrm{u(x',\cdot)\Psi(x_n-\cdot)}{s}.
\end{equation}

Applying the generalised Minkowski inequality to \eqref{upsi-ineq} 
we first obtain
\begin{equation}
    \Norm{|u(\cdot,x_n)|^s}{L_{p'/s}}\le C(R_n)
  \int |\Psi(x_n-y_n)|^s\Norm{|u(\cdot,y_n)|^s}{L_{p'/s}}\,dy_n,
\end{equation}
so by taking $q\ge1$ such that
$1+\tfrac1{r_n/s}=\tfrac1{p_n/s}+\tfrac1{q/s}$,
\begin{equation}
    \Norm{|u|^s}{L_{\vec r/s}}\le C(R_n)
  \norm{|\Psi|^s}{L_{q/s}}\norm{|u|^s}{L_{\vec p/s}}.
\end{equation}
But this means that
\begin{equation}
  \norm{u}{L_{\vec r}}^s\le c'R_n^{1-s}R_n^{s-\fracci sq}\nrm{\psi}{q}^s
  \norm{u}{L_{\vec p}}^s,
\end{equation}
and since $1-\fracc s{q}=s(\fracc1{p_n}-\fracc1{r_n})$ the claim
follows by taking roots.
\end{proof}

Using Lemma~\ref{sss-lem} and ideas from the proof of
Proposition~\ref{sob-prop}, the above theorem can now relatively easily 
be extended to a sequence
version of the Nikol$'$skij--Plancherel--Polya inequality.

This version deals with sequences $(f_j)$ in $\cal S'(\Rn)$ 
that fulfill the following spectral condition, that we could describe as a
\emph{geometric rectangle condition},
\begin{equation}
  \supp \cal F f_j \subset [-AR_1^j,AR_1^j]\times\dots\times[-AR_n^j,AR_n^j].
  \label{vNPP-cnd}
\end{equation}
Here $A>0$ is a constant, while the fixed numbers $R_1$,\dots,$R_n>1$
define the rectangles.

The next inequality will give the Sobolev embeddings in
Theorem~\ref{sob'-thm} below at once, 
but the inequality is also interesting for other purposes;
it was obtained in \cite[Prop.~2.4.1]{ScTr87} for $n=2$.

\begin{thm}
  \label{vNPP-thm}
When $\vec p\ne\vec r$ and 
$0<p_k\le r_k<\infty$ for $k=1,\dots,n$, then there is
for $0<q\le\infty$ a number $c>0$ such that
\begin{equation}
  \norm{(\sum_{j=0}^\infty |f_j(\cdot)|^q)^\fracci1q}{L_{\vec r}}
  \le c
  \Norm{\sup_{j\in \N_0} (\prod_{k=1}^nR_k^{\fracci j{p_k}-\fracci j{r_k}}
          |f_j(\cdot)|)}{L_{\vec p}}
  \label{vNPP-ineq}
\end{equation}
for all sequences $(f_j)$ in $\cal S'(\Rn)$ fulfilling \eqref{vNPP-cnd}.
\end{thm}

\begin{proof}
Using a succession of embeddings as in the proof of 
Proposition~\ref{mNPP-prop}, 
it is enough to cover the case in which
$r_m\ne p_m$ only holds for $m=n$.

Furthermore it can be assumed that $q<\min(r_1,\dots,r_n)$.
Then since $r_n<\infty $, it is a consequence of Minkowski's inequality 
that the left hand side of \eqref{vNPP-ineq} is less than 
\begin{equation}
  (\int_\R (\sum_{j=0}^\infty 
   \norm{f_j(\cdot,x_n)}{L_{p'}(\R^{n-1})}^q
   )^{\fracci{r_n}q} \, dx_n)^{\fracci1{r_n}}.
\end{equation}
To proceed we note that Lemma~\ref{sss-lem} also holds if the base of the
exponential is shifted from $2$ to $R_n>1$ (since $R_n$ is a power of $2$).
We use this version below with
$\theta:=\frac{p_n}{r_n}\in \,]0,1[\,$, and  setting
$s_0=\fracc{1}{p_n}-\fracc1{r_n}$, $s_1=-\fracc1{r_n}$,
\begin{equation}
  \theta s_0+(1-\theta)s_1= \fracc{\theta}{p_n}-\fracc1{r_n}=0.
\end{equation}
Applying this to the $\ell^0_q$-norm, 
the above integral is majorised by 
\begin{equation}
  c (\int_\R (\sup_{j} R_n^{s_0j}\norm{f_j(\cdot,x_n)}{L_{p'}})^{\theta r_n}
             \,dx_n)^{\fracci1{r_n}}
  (\sup_{j,x_n}R_n^{-\fracci j{r_n}}\norm{f_j(\cdot,x_n)}{L_{p'}})^{1-\theta}.
\end{equation}
Using Proposition~\ref{mNPP-prop} on the $L_{(r',\infty)}$-norm, 
one has for the last factor
\begin{equation}
  (\dots)^{1-\theta}\leq
  c'(\sup_j R_n^{\fracci j{p_n}-\fracci j{r_n}}
  \Norm{f_j}{L_{\vec p}})^{1-\theta}
\leq
  c'\Norm{\sup_j R_n^{\fracci j{p_n}-\fracci j{r_n}}|f_j|}
         {L_{\vec p}}^{1-\theta}.
\end{equation}
The first factor $(\dots)^{\theta/p_n}$ can be treated similarly, and
it follows that \eqref{vNPP-ineq} holds.
\end{proof}

Using Theorem~\ref{vNPP-thm} we arrive at the following Sobolev embedding.

\begin{thm}
  \label{sob'-thm}
Let $s$, $t\in \R$, $s>t$, and $p_j$, $r_j\in \,]0,\infty[$ fulfil
\begin{equation}
  s-\fracc1{p_1}-\dots-\fracc1{p_n}=t-\fracc1{r_1}-\dots-\fracc1{r_n},\qquad
  \forall k\colon r_k\ge p_k.
  \label{stpr-ineq}
\end{equation}
There is then a continuous embedding
$F^{s}_{\vec p,\infty}(\Rn)\imb F^{t}_{\vec r,q}(\Rn)$ for every $q\in
\,]0,\infty]$. 
\end{thm}
\begin{proof}
Since the ball $B(0,A\cdot2^j)$ is contained in the rectangle in
\eqref{vNPP-cnd} with $R_k=2$
for all $k$,
it suffices to take a Littlewood--Paley decomposition of an arbitrary 
$u$ in $F^s_{\vec p,\infty}(\Rn)$ 
and then insert  $f_j:=2^{tj}u_j$ into \eqref{vNPP-ineq}.
\end{proof}

In the following section, the above embedding will be extended to a set-up
with an additional anisotropy in $s$. 

\section{Spaces with mixed norms and quasi-homogeneous smoothness.}
  \label{fspqa-sect}
As is well known, it is important, say for parabolic differential equations
to consider spaces that are anisotropic in the quasi-homogeneous sense 
concerning the smoothness index $s$. This may be combined with the
mixed Lebesgue norms in the way we now describe briefly.

Each coordinate $x_j$ in $\Rn$ is given a weight $a_j\ge1$, and $\vec
a=(a_1,\dots,a_n)$; i.e.\ $\vec a=(1,\dots,1)$ is the case previously treated
in this paper. With the anisotropic dilation 
$t^{\vec a}x:=(t^{a_1}x_1,\dots,t^{a_n}x_n)$ for $t\ge0$, and 
$t^{s\vec a}x:=(t^s)^{\vec a}x$ for $s\in \R$, and in particular 
$t^{-\vec a}x=(t^{-1})^{\vec a}x$, the anisotropic distance
function $|x|_{\vec a}$ is introduced as the unique $t>0$ such
that $t^{-\vec a}x\in S^{n-1}$ (for $x\ne 0$; $|0|_{\vec a}=0$); i.e.\
\begin{equation}
  \tfrac{x_1^2}{t^{2a_1}}+\dots+  \tfrac{x_n^2}{t^{2a_n}}=1.
  \label{tax-eq}
\end{equation}
For the reader's convenience 
we recall that $|\cdot|_{\vec a}$ 
is $C^\infty$ on $\Rn\setminus\{0\}$ by the Implicit
Function Theorem; the formula $|t^{\vec a}x|_{\vec a}=t|x|_{\vec a}$ is seen
directly, and this implies the triangle inequality:
\begin{equation}
  |x+y|_{\vec a}\le |x|_{\vec a}+|y|_{\vec a}.
  \label{atriangle-ineq}
\end{equation} 
Indeed, it should be shown for $x\ne0\ne y$ 
that $1\ge \sum_{j=1}^n \tfrac{(x_j+y_j)^2}
{(|x|_{\vec a}+|y|_{\vec a})^{2a_j}}$, and since each fraction is invariant
under $(x,y)\mapsto (t^{\vec a}x,t^{\vec a}y)$, it can be
assumed that $|x|_{\vec a}+|y|_{\vec a}=1$. Then 
$0<|x|\le |x|_{\vec a}<1$, for 
\begin{equation}
  1= \sum_{j=1}^n{x_j^2}{|x|_{\vec a}^{-2a_j}}
  \ge \sum_{j=1}^n{x_j^2}{|x|_{\vec a}^{-2}}  =|x|^2|x|_{\vec a}^{-2}.
\end{equation} 
With similar results for $y$, we get
$|x+y|\le |x|+|y|\le |x|_{\vec a}+|y|_{\vec a}\le1$,
whence the inequality after \eqref{atriangle-ineq}, as desired.

As a preparation, we need the following analogues of the inequalities
between the $\ell_\infty$-, $\ell_2$- and $\ell_1$-norms on $\C^n$:
\begin{equation}
  \max(|x_1|^{1/a_1},\dots,|x_n|^{1/a_n})
  \le |x|_{\vec a}\le |x_1|^{1/a_1}+\dots+|x_n|^{1/a_n}.
  \label{ax-ineq}
\end{equation}
The inequality to the right follows from \eqref{atriangle-ineq}, 
since $|(x_1,0,\dots)|_{\vec a}=|x_1|^{1/a_1}$ etc.
By taking $t$ equal to the above maximum,
the left hand side of \eqref{tax-eq} would be $\ge1$ (the maximum is attained),
so that $t\le|\xi|_{\vec a}$.

Along with $|\cdot|_{\vec a}$,
the Littlewood--Paley decomposition is chosen 
with the modification that $\Psi_j(\xi)=\psi(2^{-j}|\xi|_{\vec a})$. Then
$\Phi_j$ is supported in the anisotropic corona
$2^{j-1}\le|\xi|_{\vec a}\le 2^{j+1}$. As usual $u_j:=\cal F^{-1}(\Phi_j\hat
u)$, but here it is understood that it is the anisotropic distance function
$|\cdot|_{\vec a}$ that goes into the construction.
 
Using this, the anisotropic Lizorkin--Triebel space
$F^{s,\vec a}_{\vec p,q}(\Rn)$, with $s\in \R$, $\vec p\in \,]0,\infty[\,^n$
and $0<q\le\infty$, consists
of the $u$ in $\cal S'(\Rn)$ having finite quasi-norms 
\begin{equation}
  \norm{u}{F^{s,\vec a}_{\vec p,q}} = 
   \norm{(\sum_{j=0}^\infty 2^{sqj}|u_j(\cdot)|^q)^\fracci1q}{L_{\vec p}}.
  \label{fspq''-eq}
\end{equation}
The corresponding Besov  space $B^{s,\vec a}_{\vec p,\infty}(\Rn)$, 
now with $0<p_k\le\infty$ for all $k$, is given by the quasi-norm
\begin{equation}
  \norm{u}{B^{s,\vec a}_{\vec p,q}} = 
   (\sum_{j=0}^\infty 2^{sqj}\norm{u_j}{L_{\vec p}}^q)^\fracci1q.
  \label{bspq''-eq}
\end{equation}
These are more precisely anisotropic spaces of the quasi-homogeneous and
mixed norm type. In the case of usual Lebesgue norms, i.e.\
$p_1=\dots=p_n=p$, such spaces have been used in general investigations, 
for example of microlocal
properties of non-linear differential equations in
\cite{Y3} and of pointwise multiplication in \cite{JJ94mlt}. 
For $1 < p_k < \infty$, $k=1, \ldots\, , n$ and positive $s$, 
anisotropic Triebel-Lizorkin spaces based on mixed $L_p$-norms 
have been investigated in the second edition of the famous book of Besov,
Il'in and Nikol$'$skij \cite{BIN96}. There they are introduced by means of
differences. Cf.\ Remark~10 below.

\bigskip

Let us first note that
$\norm{\cdot}{F^{s,\vec a}_{\vec p,q}}$ and
$\norm{\cdot}{B^{s,\vec a}_{\vec p,q}}$ are quasi-norms. Indeed,
\eqref{Lplambda-eq} implies that for $\lambda=\min(1,q,p_1,\dots,p_n)$
their powers $\nrm{\cdot}{}^{\lambda}$ are subadditive ; i.e.\ an analogue of
\eqref{Lplambda-eq} holds for them. 
So by taking $\ell_{1/\lambda}$ and 
$\ell_{(1/\lambda)^*}$-norms, $(\tfrac{1}{\lambda})^*=\tfrac{1}{1-\lambda}$,
the quasi-triangle inequality results:
\begin{equation}
  \nrm{f+g}{}\le (\nrm{f}{}^{\lambda}+\nrm{g}{}^{\lambda})^{1/\lambda}
  \le 2^{1-\lambda}(\nrm{f}{}+\nrm{g}{}).
\end{equation}
When both $q\ge1$ and all $p_k\ge1$, the spaces are therefore  
Banach spaces (in view of the completeness shown in
Proposition~9 below).  

As usual the simple embeddings $F^{s,\vec a}_{\vec p,q}\hookrightarrow
F^{s-\varepsilon,\vec a}_{\vec p,q}$ for $\varepsilon>0$, respectively
$F^{s,\vec a}_{\vec p,q}\hookrightarrow F^{s,\vec a}_{\vec p,r}$ for $q<r$
are easy to see. Similarly for $B^{s,\vec a}_{\vec p,q}$.
A direct argument and the generalised Minkowski inequality gives that
\begin{equation}
  B^{s,\vec a}_{\vec p,\min(p_1,\dots,p_n)}
\hookrightarrow   F^{s,\vec a}_{\vec p,q}
\hookrightarrow
  B^{s,\vec a}_{\vec p,\max(p_1,\dots,p_n)}.
  \label{BFB-emb}
\end{equation}
Next we show that the Sobolev embedding for these spaces
is a direct consequence of the previous inequality for distributions
fulfilling the geometric rectangle condition \eqref{vNPP-cnd}.

\begin{thm}
  \label{sob''-thm}
Let $s$, $t\in \R$, $s>t$, and $p_k$, $r_k\in \,]0,\infty]$ fulfil
\begin{equation}
  s-\fracc{a_1}{p_1}-\dots-\fracc{a_n}{p_n}=
  t-\fracc{a_1}{r_1}-\dots-\fracc{a_1}{r_n},\qquad
  \forall k\colon r_k\ge p_k.
  \label{stpra-ineq}
\end{equation}
For any $q\in\,]0,\infty]$ there are then continuous embeddings
$F^{s,\vec a}_{\vec p,\infty}\imb F^{t,\vec a}_{\vec r,q}$ 
(all $r_k<\infty$)
and $B^{s,\vec a}_{\vec p,q}\imb B^{t,\vec a}_{\vec r,q}$ between these spaces
over $\Rn$. 
\end{thm}
\begin{proof}
Since $\supp \Phi_j$ is contained in the set where $|\xi|_{\vec a}\le
2^{j+1}$, the left inequality in \eqref{ax-ineq} gives for all 
$k=1,\dots,n$ that $|\xi_k|\le 2^{a_k+ja_k}$.
Hence $\supp \Phi_j$ is a subset of the rectangle in
\eqref{vNPP-cnd} for $R_k=2^{a_k}$ and $A=2^{\max(a_1,\dots,a_n)}$.

Applying Theorem~\ref{vNPP-thm} to $f_j:=2^{tj}u_j$ for an arbitrary
$u$ in $F^{s,\vec a}_{\vec p,\infty}(\Rn)$ 
therefore yields that $\norm{u}{F^{t,\vec a}_{\vec r,q}}\le 
c\norm{u}{F^{s,\vec a}_{\vec p,\infty}}$ under the assumption in
\eqref{stpra-ineq}. The $B$-case follows from \eqref{mNPP-ineq}.
\end{proof}

\begin{rem}
In \cite[Thm.~29.10]{BIN96} embeddings of $F$-spaces into $L_{\vec{r}}$ and into Besov spaces are studied.
Our Theorem \ref{sob''-thm} supplements and improves (at least partly) the assertions stated there.
\end{rem}

Finally, as a supplement to the existing literature, we add
the next result. It is another
application of the mixed-norm
version of the Nikol$'$skij--Plancherel--Polya inequality, cf.
Proposition~\ref{mNPP-prop} above. 

\begin{prop}   \label{qB-prop}
For $s\in \R$, $0<p_k<\infty$ ($1\le k\le n$) and $0<q\le\infty$ the
Lizorkin--Triebel space $F^{s,\vec a}_{\vec p,q}(\Rn)$ is a quasi-Banach
space with continuous embeddings
\begin{equation}
  \cal S(\Rn)\hookrightarrow F^{s,\vec a}_{\vec p,q}(\Rn)
  \hookrightarrow\cal S'(\Rn).
  \label{SFS-emb}
\end{equation}
Analogous results hold for the $B^{s,\vec a}_{\vec p,q}$ spaces
with $p_k\in \,]0,\infty]$ for all $k$. 
\end{prop}
\begin{proof}
By \eqref{BFB-emb}, it suffices to show \eqref{SFS-emb} in the Besov
case; and $q=\infty$ is also enough. 
That $ \cal S(\Rn)\hookrightarrow B^{s,\vec a}_{\vec p,\infty}(\Rn)$ 
is a direct consequence of the definition, 
with a proof similar to the
isotropic one in \cite[2.3.3]{T2}.

The continuity of $B^{s,\vec a}_{\vec p,\infty}(\Rn)\hookrightarrow
\cal S'(\Rn)$ can also be carried over from \cite[2.3.3]{T2}. In so doing, the
anisotropies can be handled by Proposition~\ref{mNPP-prop} 
(like in the proof of Theorem~\ref{sob''-thm}), 
which for any $u\in B^{s,\vec a}_{\vec p,\infty}(\Rn)$ gives 
\begin{equation}
 \Norm{u_j}{L_\infty}\le c 2^{j(\fracci{a_1}{p_1}+\dots+\fracci{a_n}{p_n})}
\norm{u_j}{L_{\vec p}}. 
\end{equation}
Then one can set $\tilde\Phi_j=\Phi_{j-1}+\Phi_j+\Phi_{j+1}$, so that
$\tilde\Phi_j\equiv1$ on $\supp\Phi_j$; this gives for $u\in B^{s,\vec
a}_{\vec p,\infty}$ and $\psi\in \cal S(\Rn)$
\begin{equation}
  \begin{split}
  |\dual{u}{\overline{\psi}}|&\le \sum_{j=0}^\infty
  |\dual{u}{\overline{\cal F^{-1}(\Phi_j\tilde\Phi_j\hat\psi)}}|
\\
  &\le \sum_{j=0}^\infty
  c 2^{j(\fracci{a_1}{p_1}+\dots+\fracci{a_n}{p_n})}\norm{u_j}{L_{\vec p}}
  \norm{\cal F^{-1}(\tilde\Phi_j\hat\psi)}{L_1}
\\
  &\le c' \norm{u}{B^{s,\vec a}_{\vec p,\infty}}
     \sum_{j=0}^\infty 2^{(-s+\fracci{a_1}{p_1}+\dots+\fracci{a_n}{p_n})j}
  \norm{\cal F^{-1}(\tilde\Phi_j\hat\psi)}{L_1}.
  \end{split}
\end{equation}
Formally the last infinite series has the structure of a 
$B^{\fracci{a_1}{p_1}+\dots+\fracci{a_n}{p_n}-s}_{1,1}$-norm on $\psi$.
However, the fact that the family $(\tilde\Phi_j)$ appears instead of the
decomposition $(\Phi_j)$ is inconsequential, for the proof of the
continuity of $\cal S\hookrightarrow B^{s}_{p,q}$ also gives an estimate in
this situation. So for some seminorm $p_N$ on $\cal S$ it holds that
$|\dual{u}{\overline{\psi}}|\le c''\norm{u}{B^{s,\vec a}_{\vec p,\infty}}
p_N(\psi)$, as desired.

Completeness is shown for the $F$-case
(\cite[2.3.3]{T2} is without details on this). Given a fundamental
sequence $(u_l)$ in $F^{s,\vec a}_{\vec p,q}$, the just shown continuity
gives that $u_l-u_m$ belongs eventually to any given neighbourhood of $0$
in $\cal S'(\Rn)$; hence $(u_l)$ converges in $\cal S'$ to some $u$. So
$\check\Phi_j*u_l\to \check\Phi_j*u$ in $\cal S'$ for $l\to\infty$
(since the index $l$ refers to the sequence,  
we shall write $\check\Phi_j*u_l$ for the frequency modulations
by means of the Littlewood--Paley decomposition). 
Hence
$\check\Phi_j*u_l(x)=\dual{u_l}{\check\Phi_j(x-\cdot)}$ converges pointwisely
to $\check\Phi_j*u(x)$.

It remains to prove that $u\in F^{s,\vec a}_{\vec p,q}(\Rn)$ and 
$u_l\to u$ in the topology of this space.  For $q<\infty$ the sum over $j\in\N_0$ can be seen
as an integration w.r.t.\ the counting measure, and then
$n+1$ applications of Fatou's lemma give
$\norm{u}{F^{s,\vec a}_{\vec p,q}}\le 
\liminf\norm{u_l}{F^{s,\vec a}_{\vec p,q}}<\infty$ 
(it is convenient that positive measurable functions always have
integrals in $[0,\infty]$ and that Tonelli's theorem implies the measurability
during the successive integrations). Given $\varepsilon>0$ and
$N$ so that $\norm{u_{m}-u_l}{F^{s,\vec a}_{\vec p,q}}<\varepsilon$ for
$m,\,l>N$, similar applications of Fatou's lemma gives
\begin{equation}
  \begin{split}
    \norm{u-u_l}{F^{s,\vec a}_{\vec p,q}}
  &\le \norm{(\liminf_m
\sum_{j}2^{sjq}|\check\Phi_j*(u_{m}-u_l)|^q)^{1/q}}{L_{\vec p}}
\\
  &\le
\limsup_m \norm{u_{m}-u_l}{F^{s,\vec a}_{\vec p,q}} <\varepsilon.
  \end{split}
\end{equation}
This shows the convergence in $F^{s,\vec a}_{\vec p,q}$. For $q=\infty$ 
it is seen directly that $\sup_j|\check\Phi_j*u(x)|\le \liminf_{l}
(\sup_j |\check\Phi_j*u_l(x)|)$, and thence the conclusions follow as above.
In the $B$-cases the ingredients are the same, only with the $j$-integration
carried out last.
\end{proof}

\begin{rem}   \label{hist-rem}
To conclude we comment on the background.

Anisotropic Sobolev (or Bessel potential) spaces    
$H^{s,\vec a}_p$ and Besov spaces  $B^{s,\vec a}_{p,q}$ 
(with $1<p<\infty$ respectively $1\le p,q\le \infty$, partly with $s>0$) 
and in particular embedding relations between them 
have been investigated e.g.\ in the monographs of
Nikol$'$skij \cite{Nik75} and  
Besov, Il'in and Nikol$'$skij \cite{BIN78}, \cite{BIN96}. 
Nikol$'$skij and his co-authors departed from a definition based on derivatives and differences.
For a characterization of anisotropic Lizorkin--Triebel spaces 
$F^{s,\vec a}_{\vec p,q}$ by differences we refer to
Yamazaki \cite[Thm.~4.1]{Y2} and Seeger \cite{Se89}
(for $\vec p=(p,\dots,p)$, but general $\vec a$). 
Since Sobolev spaces $H^{s,\vec a}_p$ represent
particular cases of the Lizorkin--Triebel scale,
there is some overlap between our work and these quoted books.
In connection with anisotropic Lizorkin--Triebel spaces we would like to
mention also  
Triebel \cite{Tr77} and St\"ockert, Triebel \cite{StT}
for the Fourier-analytic characterization;
and concerning the $\varphi$-transform and characterization by atoms, we refer 
to Dintelmann \cite{Di96} and Farkas \cite{Fa00}.

As others before us, we have preferred to define the anisotropy in terms of
the function $|\cdot|_{\vec a}$.
This procedure is well known and goes back at least to the 1960's.
A list of some basic properties of $|\cdot|_{\vec a}$
can be found in \cite{Y1}, together with further historical remarks.

One advantage of using $|\cdot|_{\vec a}$ is that it gives an efficient
formalism where the powerful tools from Fourier analysis and distribution
theory are easy to invoke. 
This is clearly illustrated by the rather manageable proofs of e.g.\
Theorems~\ref{vNPP-thm} and \ref{sob''-thm}.
In general the Fourier-analytic approach gives 
streamlined, if not simpler proofs of some basic properties of the spaces.

\bigskip

Finally we note that this paper has been partly motivated by some works of
Weidemaier \cite{Wei98,Wei02} and Denk, Hieber and Pr{\"u}ess \cite{DeHiPr}, 
whose results on traces in connection with parabolic problems
can be \emph{roughly} summarised
as follows: taking traces by setting $x_1=0$ 
in the anisotropic $F^{s,\vec a}_{\vec p,q}$-spaces, with $s=2$
and $\vec a=(1,\dots,1,2)$, leads to trace spaces in the 
\emph{same} scale, and only to Besov spaces if $p_1=\dots=p_n$.

Concerning the trace problem for the full scale $F^{s,\vec a}_{\vec p,q}$,
i.e.\ with general $s$, $\vec a$, $\vec p$ and $q$,
we have corroborated this conclusion in another
joint work \cite{JoSi06}. 
With the present article our intention was to extract some 
preparations that should be of independent interest.
\end{rem}

%
%
\providecommand{\bysame}{\leavevmode\hbox to3em{\hrulefill}\thinspace}

\end{document}